\documentclass[a4paper,12pt]{article}

\usepackage{amsfonts}
\usepackage{amscd,color}
\usepackage{amsmath,amsfonts,amssymb,amscd}
\usepackage{indentfirst,graphicx,epsfig}
\usepackage{graphicx}
\input{epsf}
\usepackage{graphicx}
\usepackage{epstopdf}
\usepackage{caption}
\usepackage{subfigure}
\usepackage{mathrsfs}
\usepackage{ifpdf}

\newtheorem{thm}{Theorem}[section]

\newtheorem{lem}[thm]{Lemma}
\newtheorem{ob}[thm]{Observation}
\newtheorem{claim}[thm]{Claim}

\newtheorem{pro}[thm]{Proposition}
\newenvironment {proof} {\noindent{\em Proof.}}{\hspace*{\fill}$\Box$\par\vspace{4mm}}
\newcommand{\ml}{l\kern-0.55mm\char39\kern-0.3mm}

\title{\textbf{ Proper disconnection of graphs\footnote{Supported by NSFC No.11871034, 11531011 and NSFQH No.2017-ZJ-790.}}}
\author{{\small Xuqing Bai$^1$, You Chen$^1$, Meng Ji$^1$, Xueliang Li$^{1,2}$, Yindi Weng$^1$, Wenyan Wu$^1$} \\
{\small  $^1$Center for Combinatorics and LPMC}\\
{\small Nankai University, Tianjin 300071, China}\\
\small $^2$School of Mathematics and Statistics, Qinghai Normal University\\
\small Xining, Qinghai 810008, China\\
{\small baixuqing0@163.com, chen\_you@163.com, jimengecho@163.com} \\
{\small lxl@nankai.edu.cn, 1033174075@qq.com, wuwenyanhn@163.com}\\
}
\date{}
\begin{document}
\maketitle
\begin{abstract}
For an edge-colored graph $G$, a set
$F$ of edges of $G$ is called a \emph{proper cut} if $F$ is an edge-cut of $G$
and any pair of adjacent edges in $F$ are
assigned by different colors. An edge-colored
graph is \emph{proper disconnected} if for each pair of distinct
vertices of $G$ there exists a proper edge-cut separating them.
For a connected graph $G$, the \emph{proper disconnection
number} of $G$, denoted by
$pd(G)$, is the minimum number of
colors that are needed in order to make $G$ proper
disconnected. In this paper, we first give the exact values of the
proper disconnection numbers for some special families of
graphs. Next, we obtain a sharp upper bound of
$pd(G)$ for a connected graph $G$ of order $n$, i.e,
$pd(G)\leq \min\{ \chi'(G)-1, \left \lceil \frac{n}{2} \right \rceil\}$.
Finally, we show that for given integers $k$ and $n$, the
minimum size of a connected graph $G$ of order $n$
with $pd(G)=k$ is $n-1$ for $k=1$ and
$n+2k-4$ for $2\leq k\leq \lceil\frac{n}{2}\rceil$.
\\[2mm]
\end{abstract}

\section{Introduction}

All graphs considered in this paper are simple, finite and
undirected. Let $G=(V(G),E(G))$ be a nontrivial
connected graph with vertex set $V(G)$ and edge set $E(G)$.
For $v\in V(G)$, let $d(v)$ denote the $degree$ of $v$, $N(v)$ denote the $neighborhood$
of $v$, and $N[v]$ denote the closed neighborhood of $v$ in $G$. Let $S$ be a subset of $V(G)$, denote by $N(S)$ the
set of neighbors of $S$ in $G$. Denote the diameter of $G$ by $D(G)$. For any notation
or terminology not defined here, we follow those used in \cite{BM}.

Throughout this paper, we use $P_n$, $C_n$, $K_n$ to denote the path,
the cycle and the complete graph of order $n$, respectively. Given two
disjoint graphs $G$ and $H$, the \emph{join} of $G$ and
$H$, denoted by $G\vee H$, is obtained from the vertex-disjoint
copies of $G$ and $H$ by adding all edges between $V(G)$ and $V(H)$.

For a graph $G$, let $c: E(G)\rightarrow[k] = \{1,2,...,k\}$,
$k \in \mathbb{N}$, be an edge-coloring of $G$.
For an edge $e$ of $G$, we denote the color of $e$
by $c(e)$. When adjacent edges of $G$ receive different
colors by $c$, the edge-coloring $c$ is called \emph{proper}.
The \emph{chromatic index} of $G$, denoted by $\chi'(G)$,
is the minimum number of colors needed in a proper edge-coloring of $G$.

Chartrand et al. in \cite{CDHHZ} introduced the concept
of rainbow disconnection of graphs.
An \emph{edge-cut} of a graph $G$ is a set $R$ of edges such that
$G-R$ is disconnected. An edge-coloring is called a
\emph{rainbow disconnection coloring} of $G$ if for every
two vertices of $G$, there exists a rainbow cut in
$G$ separating them. For a connected graph $G$, the \emph{rainbow disconnection number}
of $G$, denoted $rd(G)$, is the smallest number of
colors required for a rainbow disconnection coloring of
$G$. A rainbow disconnection
coloring with $rd(G)$ colors is called an
rd-\emph{coloring} of $G$. In \cite{BCL,BL,HL} the authors have
obtained many results.

Inspired by the concept of rainbow disconnection,
we naturally put forward a concept of proper
disconnection. For an edge-colored
graph $G$, a set $F$ of edges of $G$ is a \emph{proper cut}
if $F$ is an edge-cut of $G$ and any pair of adjacent edges
in $F$ are assigned by different colors. An edge-colored
graph is called \emph{proper disconnected} if there
exists a proper cut for each pair of distinct vertices
of $G$ separating them. For a connected graph $G$, the
\emph{proper disconnection number} of $G$, denoted by $\mathit{pd(G)}$, is
defined as the minimum number of colors that are needed in order to make $G$
proper disconnected, and such an edge-coloring is called a
pd-\emph{coloring}. Clearly, for any pair of vertices
of a graph, a rainbow cut is definitely a proper cut.
In \cite{CDHHZ}, we obtained that if $G$ is a nontrivial
connected graph, then
$\lambda(G) \leq \lambda^+(G) \leq rd(G) \leq \chi'(G) \leq \Delta(G)+1$.
Hence, we immediately have the following observation.
\begin{ob}\label{observation}
If $G$ is a nontrivial connected graph, then
$1\leq \mathit{pd}(G)\leq \mathit{rd}(G)\leq \chi'(G)\leq \Delta(G)+1$.
\end{ob}

The concept of monochromatic disconnection of graphs was also put forward by us.
For details we refer to \cite{PL1, PL2}.

\section{Preliminaries}

At the very beginning, we state some fundamental
results on the proper disconnection number
of graphs, which will be used in the sequel.

For trees and cycles, we have the following easy results.

\begin{pro}\label{tree}
If $G$ is a tree, then $pd(G)=1$.
\end{pro}

\begin{pro}\label{cycle}
If $C_n$ be a cycle, then
$$\textnormal{pd}(C_n)=\left\{
\begin{array}{lcl}
2, &  {if~n=3},\\
1, &  {if~n\geq 4.}
\end{array}
\right.$$
\end{pro}

\begin{lem}\label{4}
If $H$ is a connected subgraph of a connected graph $G$, then
$pd(H)\leq pd(G)$.
\end{lem}

\begin{proof} Let $c$ be a pd-coloring
of $G$. Let $x$ and $y$ be two vertices of $G$. Suppose that $F$ is an $x-y$ proper cut in $G$. Then $F\cap E(H)$ is an $x-y$ proper cut in $H$. Hence, $c$ restricted to $H$ is a proper disconnection coloring of $H$. Thus, $pd(H)\leq pd(G)$.
\end{proof}

A \emph{block} of a graph $G$ is a maximal connected subgraph of $G$ that has no cut-vertex.
Then the block is either a cut-edge, say trivial block, or a maximal 2-connected subgraph. Let $\{B_1,B_2, ...,B_t\}$ be the block set of $G$.

\begin{lem}\label{5}
Let $G$ be a nontrivial connected graph. Then $\mathit{pd}(G)=\max\{ \mathit{pd}(B_i) \ | \ B_i,i\in[t]\}$.
\end{lem}

\begin{proof}
Let $G$ be a nontrivial connected
graph. Let $\{B_{1},B_{2},\cdots,B_{t}\}$ be a block
decomposition of $G$, and let
$k=\max\{pd(B_{i}) \ | \ 1\leq i \leq t\}$. If $G$ has
no cut vertex, then $G=B_{1}$ and the result follows.
Next, we assume that $G$ has at least one cut vertex.
Since each block is a subgraph of $G$,
$pd(G)\geq k$ by Lemma \ref{4}.

Let $c_i$ be a pd-coloring of $B_i$. We define
the edge-coloring $c:$ $E(G) \rightarrow [k]$
of $G$ by $c(e)=c_i(e)$ if $e\in E(B_i)$.

Let $x,y\in V(G)$. If there exists a block,
say $B_i$, that contains both $x$ and $y$,
then any $x-y$ proper cut in $B_i$ is an $x-y$
proper cut in $G$. Next, we consider that no
block of $G$ contains both $x$ and $y$. Assume
that $x\in B_i$ and $y\in B_j$, where $i\neq j$.
Now every $x-y$ path contains a cut vertex,
say $v$, of $G$ in $B_i$ and a cut vertex,
say $w$, of $G$ in $B_j$. Note that $v$ could
equal $w$. If $x\neq v$, then any $x-v$ proper
cut of $B_i$ is an $x-y$ proper cut in $G$.
Similarly, if $y\neq w$, then any $y-w$ proper
cut of $B_j$ is an $x-y$ proper cut in $G$.
Thus, we may assume that $x=v$ and $y=w$.
It follows that $v\neq w$. Consider the $x-y$
path $P=(x=v_1,v_2,\dots,v_p=y)$. Since $x$
and $y$ are cut vertices in different blocks
and no block contains both $x$ and $y$, we can
select the first cut vertex $z$ of $G$ on $P$
except $x$, that is, $z=v_k$ for some
$k$ $(2\leq k\leq p-1)$. Then $x$ and $z$
belong to the same block, say
$B_s$ $(s\in \{1,2,\cdots,t\}\backslash \{i,j\})$.
Then any $x-z$ proper cut of $B_{s}$ is an $x-y$
proper cut of $G$. Hence, $pd(G)\leq k$, and
so $pd(G)=k$.
\end{proof}

An edge cut $F$ of $G$ is a \emph{matching cut}
if $F$ is a matching of $G$. For a
vertex $v\in V(G)$, let $E_{v}$ be the set of
all the edges incident with $v$ in $G$.

\begin{thm}\label{pd23}
Let $G$ be a nontrivial connected graph.
Then $pd(G)=1$
if and only if for any two vertices
of $G$, there exists a matching cut separating them.
\end{thm}
\begin{proof}
Let $pd(G)=1$. Assume, to the contrary, that
there exist two vertices $x$ and $y$ which
have no matching cut, i.e. each $x-y$ proper
cut has two adjacent edges. Obviously, the
two adjacent edges are colored differently.
That is, $pd(G)\geq 2$. It is a contradiction.

For the converse, define a coloring $c$ such
that $c(e)=1$ for every $e\in E(G)$. For any
two vertices $x$ and $y$ in $G$, there is a
matching cut which is an $x-y$ proper cut.
Thus, $c$ is a proper disconnection coloring
of $G$ and so $pd(G)=1$.
\end{proof}

From this result, one can see that for a hypercube $Q_n$,
$pd(Q_n)=1$.

\begin{lem}\label{l1}
Let $G$ be a nontrivial connected graph. If there
exist two vertices $u$ and $v$ sharing $t$ $(t\geq1)$
common neighbors in $G$, then
$pd(G)\geq \lceil \frac{t}{2}\rceil$.
Furthermore, if $uv\in E(G)$, then
$pd(G)\geq \lceil \frac{t}{2}\rceil+1$.
\end{lem}

\begin{proof}
Let $c$ be a pd-coloring of $G$ and $F(u,v)$
be a $u-v$ proper cut under $c$. Suppose that
$W=N(u)\cap N(v)=\{w_{1},w_{2},\cdots,w_{t}\}$.
Then there are $t$ internally disjoint paths of
length two. Let $E_{1}=\{uw_{i}| 1\leq i\leq t\}$
and $E_{2}=\{vw_{i}| 1\leq i\leq t \}$.
Then
$|F(u,v)\cap E_{1}| \geq  \lceil \frac{t}{2}\rceil$
or $|F(u,v)\cap E_{2}| \geq   \lceil \frac{t}{2}\rceil$.
Otherwise there exists at least one $u-v$ path of
length two in $G\setminus F(u,v)$, a contradiction.
Since $E_{1}\subseteq E_{u}$ and
$E_{2}\subseteq E_{v}$, $pd(G)\geq  \lceil \frac{t}{2}\rceil$.
Moreover, if $uv\in E(G)$, then $uv\in F(u,v)$. So
$|F(u,v)\cap E_{u}|\geq |F(u,v)\cap E_{1}| +1$ or
$|F(u,v)\cap E_{v}|\geq |F(u,v)\cap E_{2}| +1$.
Hence, $pd(G)\geq \lceil \frac{t}{2}\rceil+1$.
\end{proof}

\section{Main results}

In this section, we give the exact values of the proper disconnection
numbers for the wheel graphs, the complete graphs,
the complete bipartite graphs and the outerplanar
graphs. Furthermore, we obtain a sharp upper bound of
$pd(G)$, and derive the minimum size of a graph $G$ of
order $n$ with $pd(G)=k$, where
$1\leq k\leq \lceil\frac{n}{2}\rceil$.
\subsection{Wheel graphs}

\begin{lem}\label{pd-k4}
Let $G=K_4-\{e\}$. Then $pd(G)=2$. Furthermore,
the colors of matching edges are the same for
any pd-coloring.
\end{lem}
\begin{proof}
Let $V(G)=\{v_1, v_2, v_3,v_4\}$ and
$E(G)=\{e_{i,j}|1\leq i< j\leq 4\}\setminus\{e_{2,4}\}$.
First, we have $pd(G)\geq 2$ since $K_3$ is a
subgraph of $G$.
Define an edge coloring $c$: $E(G)\rightarrow [2]$
of $G$ as follows. Let
$c(e_{1,2})=c(e_{1,3})=c(e_{3,4})=1$,
$c(e_{1,4})=c(e_{2,3})=2$.
Let $u$ and $v$ be two vertices of $G$.
If $d(u)=2$ or $d(v)=2$, then the edge set incident
with $u$ (or $v$) is a $u-v$ proper cut.
Otherwise,
If $d(u)=d(v)=3$, then the edge set
$\{v_1v_4,v_1v_3,v_2v_3\}$ is a proper cut of
$u$, $v$. Thus, $pd(G)\leq 2$.

For any a pd-coloring, assume that $c(e_{1,3})=1$.
Since $pd(G)=2$ and $v_1-v_3$ cut has at least
three edges, there exist two matching edges
respectively incident with $v_1$, $v_3$ having
color $2$. We claim that the remaining two matching
edges must have the same color. Otherwise, there
has a vertex, say $v_2$, the colors of $E_{v_2}$ are $2$. $E_{v_1}$ or $E_{v_3}$ has two
edges with color $2$. Then there is no $v_1-v_2$
or $v_{2}-v_{3}$ proper cut.
\end{proof}

\begin{thm}\label{W_n}
If $W_n=C_n\vee K_1$ is the wheel of order $n+1\geq 4$,
then
$$\textnormal{pd}(W_n)=\left\{
\begin{array}{lcl}
2, &  {if~n=3k},\\
3, &  {otherwise.}
\end{array}
\right.$$
\end{thm}

\begin{proof}
Let $V(W_n)=\{v_0,v_1,\ldots,v_n\}$ and
$E(W_n)=\{v_0v_i,v_0v_n,v_iv_{i+1},
v_1v_n|1\leq i\leq n-1\}$. Let $e_{i,j}={v_iv_j}$.
For convenience, we use the elements of
$\mathbb{Z}_{n}$ of integer mod n to express subscripts.
First, pd$(W_n)\geq 2$, since $K_3$ is a subgraph of $W_n$.

\emph{Case 1}. If $n=3k$, then define an edge
coloring $c$: $E(W_n)\rightarrow [2]$ of $W_n$.
Let $c(e_{0,3i})=c(e_{1+3j,2+3j})=2$ where
$1 \leq i\leq k$ and $0\leq j\leq k-1$ and assign
color $1$ to the remaining edges.
Let $v_i$, $v_j$ and $v_r$ be vertices of $W_n$.
Consider $i=3t,j=3t+1,r=3t-1$ $(1\leq t\leq k)$,
then the edge set
$\{e_{i-1,i},e_{0,i},e_{0,i+1},e_{i+1,i+2}\}$
is a proper cut between $\{v_i,v_j\}$ and
$V(W_n)\setminus\{v_i,v_j\}$ and the edge
set $\{e_{i,i+1},e_{0,i},e_{0,i-1},e_{i-2,i-1}\}$
is also a proper cut between $\{v_i,v_r\}$ and
$V(W_n)\setminus\{v_i,v_r\}$.
Let $v_k$, $v_l$ be any two vertices of $W_n$
where $k$, $l$ are integers.
If $v_k$ is not adjacent to $v_l$,
then there exists an edge $v_kv_p$ such that
$p=3t$ or $k=3t$. By above argument, we have a
proper cut between $\{v_kv_p\}$
and $V(W_n)\setminus\{v_k,v_p\}$, which is a
proper cut between $v_k$ and $v_l$.
Assume $v_k$ is adjacent to $v_l$ with $k\leq l$.
If $k$ or $l$ is multiple of $3$, without loss of
generality, $k=3t$, then there exists a proper
cut between $\{v_k,v_p\}$ and $V(W_n)\setminus\{v_k,v_p\}$
where $v_q\in N(v_k)\setminus \{v_{0},v_{l}\}$,
which is a proper cut between $v_k$ and $v_l$.
If $k$, $l$ are not multiple of $3$, then there
exits a proper cut between $\{v_k,v_s\}$ and
$V(W_n)\setminus\{v_k,v_s\}$ where
$v_s\in N(v_k)\setminus \{v_{0},v_{l}\}$,
which is a proper cut between $v_k$ and $v_l$.
Thus, pd$(W_n)=2$.

\emph{Case 2}. $n\neq 3k$.
Assume that pd$(W_n)=2$. Let $c(e_{0,1})=1$.
Then $c(e)=1$ for any edge $e$ of $G$ by Lemma
\ref{pd-k4}. This is a contradiction with pd$(W_n)\geq 2$.
Thus, pd$(W_n)\geq 3$.

Now we define an edge coloring $c$ :
$E(G)\rightarrow [3]$ of $W_{n}$ $(n\neq 3k)$.
First, let $c$ be a proper edge coloring of $C_n$
using the colors $1,2,3$. For each integer $i$
with $1\leq i\leq n$, let
$a_i\in \{1,2,3\}\setminus\{c(v_{i-1}v_i),c(v_{i}v_{i+1})\}$,
and let $c(v_0v_i)=a_i$.
Thus, $E_{v_i}$ is a proper set for $1\leq i\leq n$.
Let $x$, $y$ be two distinct vertices of $W_n$.
Then at least one of $x$ and $y$ belongs to $C_n$,
say $x\in V(C_n)$. Since $E_x$ separates $x$ and
$y$, it follows that $c$ is a proper disconnection
coloring of $W_n$ using three colors.
Therefore, $pd(G)=3$ for $n\neq 3k$.
\end{proof}
\subsection{Complete bipartite graphs and complete graphs}
Now we introduce some notations. Let $X$ and $Y$
be sets of vertices of a graph $G$, we denote by
$E[X,Y]$ the set of all the edges of $G$ with
one end in $X$ and the other end in $Y$.

\begin{thm}\label{pd(K_{n,n})}
Let $K_{n,n}$ be a complete bipartite graph.
Then $pd(K_{n,n})=\lceil\frac{n}{2} \rceil$.
\end{thm}
\begin{proof}
Let $G=K_{n,n}$ and suppose that $X$ and $Y$ are
two partite vertex sets of $G$, where
$X=\{x_{1}, x_{2}, \cdots, x_{n}\}$ and
$Y=\{y_{1}, y_{2}, \cdots, y_{n}\}$. For any
two vertices $v_{i},v_{j}\in X$, there are $n$
common neighbors in $Y$. From Lemma \ref{l1},
it follows that $pd(G)\geq \lceil\frac{n}{2} \rceil$.

Now, for the upper bound, we define an edge
coloring $c: E(G) \rightarrow \{0,1,\cdots,\lceil\frac{n}{2} \rceil-1\}$
of $G$ by assigning each edge $x_{i}y_{j}$
with $c(x_{i}y_{j}) \equiv i+j-1 \pmod {\lceil\frac{n}{2} \rceil}$ for $1\leq i, j\leq n$.
Let $X=X_{1}\cup X_{2}$,
where $X_{1}=\{x_{1}, x_{2}, \cdots, x_{\lceil\frac{n}{2} \rceil}\}$
and $X_{2}=\{x_{\lceil\frac{n}{2} \rceil+1}, x_{\lceil\frac{n}{2} \rceil+2}, \cdots, x_{n}\}$. Let  $Y=Y_{1}\cup Y_{2}$,
where $Y_{1}=\{y_{1}, y_{2}, \cdots, y_{\lceil\frac{n}{2} \rceil}\}$ and $Y_{2}=\{y_{\lceil\frac{n}{2} \rceil+1}, y_{\lceil\frac{n}{2} \rceil+2}, \cdots, y_{n}\}$.

\begin{claim} For each pair of vertex sets $ X_{k}$
and $Y_{l}$ $(k,l\in[2])$, $E(X_{k},Y_{l})$ is a proper set.
\end{claim}

\begin{proof}
By symmetry, we only consider the vertex sets $X_{1}$
and $Y_{1}$. For each vertex $x_{i}$ of $X_{1}$,
since
$c(x_{i}y_{j})\equiv i+j-1 \pmod {\lceil\frac{n}{2} \rceil}$,
it follows that
$c(x_{i}y_{1})\neq c(x_{i}y_{2})\neq \cdots
\neq c(x_{i}y_{\lceil\frac{n}{2} \rceil})$.
Therefore, the edges of $G[X_{1},Y_{1}]$ incident
with the same vertex are colored by different colors.
Thus, $E(X_{1},Y_{1})$ is a proper set.
 \end{proof}
 We now show that for each pair of vertices $u$ and
 $w$ of $G$, there is a proper cut separating them.
 Two cases are needed to be discussed.

\emph{Case 1. $u\in X$, $w\in Y$. }

Suppose that $u=x_{i}$ and $w={y_{j}}$.
If $x_{i}\in X_{1}$, $y_{j}\in Y_{1}$ or
$x_{i}\in X_{2}$, $y_{j}\in Y_{2}$.
Let $F(x_{i},y_{j})$= $E[X_{1},Y_{1}] \cup E[X_{2},Y_{2}]$.
Consider the subgraph $H$ of $G$ obtained by
deleting $F(x_{i},y_{j})$ from $G$, then $H$
has two components $G[X_{1},Y_{2}]$ and
$G[X_{2},Y_{1}]$ (See Figure \ref{f4}).
Since $x_{i}\in G[X_{1},Y_{2}]$ and $y_{j}\in G[X_{2},Y_{1}]$,
we can know that $F(x_{i},y_{j})$ is an edge
cut separating $x_{i}$ and $y_{j}$.
Moreover, $E[X_{1},Y_{1}]$ and $E[X_{2},Y_{2}]$
are proper sets by the claim and
$E[X_{1},Y_{1}]\cap E[X_{2},Y_{2}]= \phi$,
so $F(x_{i},y_{j})$ is an $x_{i}-y_{j}$
proper cut. If $x_{i}\in X_{1}$, $y_{j}\in Y_{2}$ or
$x_{i}\in X_{2}$, $y_{j}\in Y_{1}$,
we can similarly show that
$E[X_{1},Y_{2}] \cup E[X_{2},Y_{1}]$
is an $x_{i}-y_{j}$ proper cut.

\begin{figure}[!htb]
  \centering
  \includegraphics[width=0.4\textwidth]{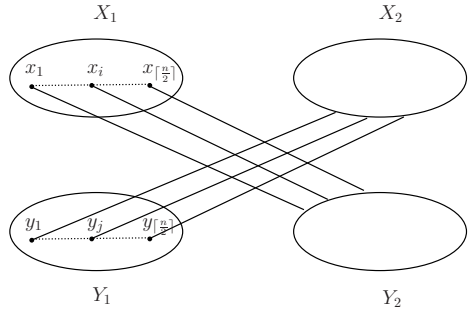}
  \caption{A graph $H$}\label{f4}.
\end{figure}

\emph{Case 2. $u,w\in X$ or $u,w\in Y$.}

By symmetry, suppose that $u=x_{i}$ and $w=x_{j}$.
If $x_{i},x_{j}\in X_{1}$,
where $1\leq i < j \leq \lceil\frac{n}{2} \rceil$.
Let $X'_{1}=X_{1}-x_{i}+v_{i+\lceil\frac{n}{2} \rceil}$ and  $X'_{2}=X_{2}+x_{i}-v_{i+\lceil\frac{n}{2} \rceil}$,
then $x_{i}\in X'_{2}$ and $x_{j}\in X'_{1}$.
For every vertex $y_{t}$ of $Y$, we know that
$c(x_{i}y_{t})=c(x_{i+\lceil\frac{n}{2} \rceil}y_{t})
\equiv i+t-1 \pmod {\lceil\frac{n}{2} \rceil}$ .
By the same method used in the claim,
for vertex sets $ X'_{k}$ and
$Y_{l}$ $(k,l\in[2])$, $E(X'_{k},Y_{l})$ is a proper
set. Let $F'(x_{i},x_{j})=E(X'_{1},Y_{1}) \cup E(X'_{2},Y_{2})$.
Let $H'= G[X_{1}',Y_{2}] \cup G[X_{2}',Y_{1}]$ (See Figure \ref{f2})
by deleting $F'(x_{i},x_{j})$ from $G$.
According to Case 1, $F'(x_{i},x_{j})$ is an
$x_{i}-x_{j}$ proper cut.
If $x_{i}\in X_{1}$ and $x_{j}\in X_{2}$.
Similarly, we can get $E[X_{1},Y_{1}] \cup E[X_{2},Y_{2}]$
is an $x_{i}-x_{j}$ proper cut.
\end{proof}

\begin{figure}[!htb]
  \centering
  \includegraphics[width=0.4\textwidth]{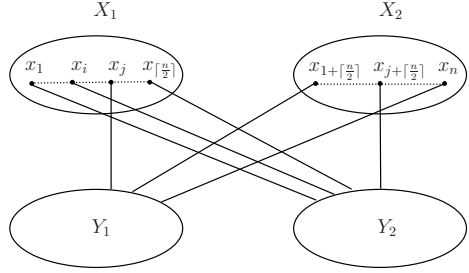}
  \caption{A graph $H'$}\label{f2}.
\end{figure}

\begin{thm}\label{K_{m,n}}
Let $K_{m,n}$ be a complete bipartite graph
with $2\leq m\leq n $. Then
$\mathit{pd}(K_{m,n})=\lceil\frac{n}{2}\rceil$.
\end{thm}
\begin{proof}
Let $G=K_{m,n}$ and $V(G)=X\cup Y$, where
$X=\{x_{1},\cdots, x_{m}\}$ and
$Y=\{y_{1}, \cdots, y_{n}\}$. Consider any
two vertices $v_{i},v_{j}\in X$, there are
$n$ common neighbors in $Y$. From Lemma
\ref{l1}, it follows that
$pd(G)\geq \lceil\frac{n}{2} \rceil$.
Also, $G$ is a subgraph of $K_{n,n}$.
By Theorem \ref{pd(K_{n,n})} and Lemma
\ref{4}, we have that
$pd(G)\leq pd(K_{n,n})= \lceil\frac{n}{2} \rceil$.
Hence, it follows that
$pd(G)= \lceil\frac{n}{2} \rceil$.
\end{proof}

\begin{thm}\label{t2}
For each integer $n\geq 2$, $pd(K_{n})=\lceil \frac{n}{2} \rceil$.
\end{thm}

\begin{proof}
Let $a=\lceil \frac{n}{2}\rceil$ and
$V(K_{n})=X\cup Y$, where $X=\{v_{1},v_{2},\cdots,v_{a}\}$
and $Y=\{v_{a+1},v_{a+2},\cdots,v_{n}\}$.
For any two vertices $v_{i},v_{j}\in V(K_{n})$,
there are $n-2$ common neighbors in $K_{n}$
and $v_{i}v_{j}\in E(K_{n})$. Then
$pd(K_{n})\geq \lceil\frac{n}{2}\rceil$
by Lemma \ref{l1}. Define an edge coloring
$c: E(G)\rightarrow \{0,1,\cdots,a-1\}$ such
that $c(v_{i}v_{j})\equiv i+j-1 \pmod a$.
For any two vertices $u$, $w\in V(G)$,
let $u\in X$ and $w\in Y$. Since $E[X,Y]$
separates $u$ and $v$, and $E[X,Y]$ is a
$u-w$ proper set in $K_{n}$ by a similar
proof method as the claim of above theorem,
then $E[X,Y]$ is a $u-w$ proper cut.
If $u,w \in X$, assume that $u=v_{i}$ and
$w=v_{j}$, where $i\leq j$. Let $i_{0}=i+a$,
then $v_{i_{0}}\in Y$.
Let $v_{r}\in V(K_{n})(r\neq i,j)$, then
$c(v_{i}v_{r})=c(v_{i_{0}}v_{r})$.
Let $X'=X\setminus \{v_{i}\}$ and
$Y'=Y\setminus \{v_{i_{0}}\}$.
Similarly,
$F(v_i,v_j)=E[X'\cup\{v_{i_{0}}\},Y\cup \{v_{i}\}]$
is a proper set. Since $F(v_i,v_j)$ separates $u$
and $w$, $F(v_i,v_j)$ is a $u-w$ proper cut in
$K_{n}$. If $u,w \in Y$, we can obtain a $u-w$
proper cut in the similar way. Thus, $c$ is a
proper disconnection coloring of $K_{n}$ and
so $pd(K_{n})\leq \lceil\frac{n}{2}\rceil$.
\end{proof}

\subsection{Outerplanar graphs}

An \emph{outerplanar graph} is a graph
that can be drawn in the plane without
crossings in such a way that all of the
vertices belong to the unbounded face of
the drawing.
A \emph{minor} of a graph $G$ is any
graph obtained from $G$ by means of a
sequence of vertex and edge deletions and
edge contractions.
There is a characterization
of outerplanar graphs as follows.

\begin{thm}\upshape\cite{D}\label{defineouterplanar}
A graph is outerplanar if and only if
it does not contain $K_4$ or $K_{2,3}$
as a minor.
\end{thm}

Let $v$ be an isolated vertex and $P_n$
a path of order $n$. We denote
$v\vee P_n$ by $F_{1,n}$, called
\emph{fan graphs}. We first show
that the proper disconnection number of
fan graph is 2, which will be used to
characterize the outerplanar graphs with
diameter 2.

\begin{lem}\label{pd(fan)}
$\mathit{pd}(F_{1,n})=2$.
\end{lem}
\begin{proof}
Let $P_n=v_1v_2\cdots v_n$ and $v_{0}$
be an isolated vertex. Let $F_{1,n}=v_{0}\vee P_n$.
Clearly, $F_{1,n}$ is a subgraph of $W_{n}$.
Firstly, we have that
$\mathit{pd}(F_{1,n})\geq\mathit{pd}(K_3)=2$
by Lemma \ref{4}. If $n=3k \ (k\geq 1)$,
then
$\mathit{pd}(F_{1,n})\leq \mathit{pd}(W_{n})=2$
by Lemma \ref{4} and Theorem \ref{W_n}.
Thus, $F_{1,3k}=2$. If $n=3k-2$ or $3k-1\ (k\geq 1)$,
then $\mathit{pd}(F_{1,3k-2})\leq \mathit{pd}(F_{1,3k-1})\leq \mathit{pd}(F_{1,3k})=2$
since $F_{1,3k-2}$ and $F_{1,3k-2}$ are
subgraphs of $F_{1,3k}$.
Hence, $\mathit{pd}(F_{1,n})=2$.
\end{proof}

\begin{thm}\label{outerplanar}
Let $G$ be an outerplanar graph. Then
$\mathit{pd}(G)=1$ if and only if $G$
is a triangle-free graph.
\end{thm}
\begin{proof}
Let $\mathit{pd}(G)=1$. Assume, to the
contrary, that $G$ contains a copy of
$K_3$.
Then $\mathit{pd}(G)\geq \mathit{pd}(K_3)=2$
by Theorem \ref{4}, a contradiction. Then
$G$ is a triangle-free outerplanar graph.

For the converse, let $B$ be the maximum
block of $G$ with $t$ vertices. Then $B$ is
a triangle-free outerplanar graph. By Lemma
\ref{5}, it turns to show that $pd(B)=1$.
If $B$ is trivial, then $pd(B)=1$. If $B$
is not trivial, then $t\geq 4$. Suppose that
$B$ is a cycle, $pd(B)=1$ by Proposition
\ref{cycle}. So it remains to consider that
$B$ is not unicyclic and $t\geq 6$. We
proceed by induction on $t$. When $t= 6$,
it has only one chord. Obviously, it is true.
When $t\geq7$,
let $C=v_1e_1v_2e_2\cdots e_{t-1}v_te_tv_1$
be the boundary of the outer face in $B$.
Choose a chord $v_iv_j$ such that the internal vertices
of $P=v_ie_iv_{i+1}e_{i+1}\cdots e_{j-1}v_j(j\geq i+3)$
have degree 2 in $B$. We use the elements of $\mathbb{Z}_{n}$ of
integer mod $n$ to express subscripts.
Let $C'=P+\{v_iv_j\}$ and $B'$ be a graph
by removing internal vertices of $P$ from $B$.
Then $pd(B')=1$ by induction hypothesis.
For any two vertices $u$, $v$ in $B$, if $u$,
$v\in B'$, there is a $u-v$ matching cut
$F_{B'}(u,v)$ in $B'$ by Theorem \ref{pd23}.
Then $F_{B'}(u,v)\cup \{e_{i+1}\}$ is a $u-v$
matching cut in $B$. If $u$, $ v\in B\setminus B'$,
then $F_{B'}(v_i,v_j)\cup \{e\}$ is a $u-v$
matching cut in $B$ where $e$ is an $u-v$ edge
cut in $P$. If $u\in B'$, $v\in B\setminus B'$,
then $\{e_i,e_{j-1}\}$ is a $u-v$ matching cut
in $B$. Therefore, $pd(B)=1$.
\end{proof}

Now we characterize the outerplanar graphs
$G$ with $D(G)=2$. We first construct
some graph classes. Let $\mathcal{D}$ be a family of graphs
obtained from $W_n=K_1\vee C_n$ by deleting
$t$ $(1\leq t\leq n-2)$ edges of $C_n$.
Let $z$ be an isolated vertex
and $v_1v_2v_3v_4v_5$ be a path of length 4.
Then join $z$ with $v_1,v_2,v_4$ and $v_5$.
We denote the resulting graph by $F^-_{1,5}$.
Let $y$ be an isolated vertex and $v_1v_2v_3v_4$
be a path of length 3. Then join $y$ with $v_1,v_3$
and $v_4$. We denote the resulting graph by
$F^-_{1,4}$. We denote $v_1v_2v_3v_4v_5v_6v_1$
by $C_6$.
Then let $F'=C_6\cup\{v_1v_3,v_3v_5,v_1v_5\}$.

\begin{thm}\label{diam}
Let $G$ be an outerplanar graph with $D(G)=2$.
Then $\mathit{pd}(G)=2$ if and only if
$G\in \mathcal{D}$ or $G\cong F^-_{1,5}$
or $F^-_{1,4}$ or $F'$.
\end{thm}
\begin{proof}
\emph{Sufficiency.} Since there is at least
one triangle $K_3$ for every $G\in \mathcal{D}$,
it is clear to see that
$\mathit{pd}(G)\geq\mathit{pd}(K_3)=2$ by
Theorem \ref{4} and Theorem \ref{t2}.
Meanwhile, $G$ is a subgraph of a fan graph,
therefore, $\mathit{pd}(G)\leq 2$ by Lemma
\ref{pd(fan)} and Lemma \ref{4}.
Hence, $\mathit{pd}(G)=2$ for $G\in \mathcal{D}$.
Similarly, $\mathit{F^-_{1,5}}=\mathit{F^-_{1,4}}=2$.
For the graph $F'$, $\mathit{pd}(F')\geq2$.
We now assign a 2-edge-coloring
$c:$ $E(F')\rightarrow \{1,2\}$ for $F'$.

Let $c(v_1v_5)=c(v_3v_4)=c(v_2v_3)=2$ and the
remaining edges are colored by 1. Thus, for
every pair of vertices in $F'$, there exists
a proper cut $F[V_1,V(F')\setminus V_1]$, where
$V_1=\{v_1,v_2,v_3\}$ and $V(F')\setminus
V_1=\{v_4,v_5,v_6\}$ or $V_1=\{v_1,v_2,v_6\}$
and $V(F')\setminus V_1=\{v_4,v_5,v_3\}$ or
$V_1=\{v_1,v_2,v_3,v_5,v_6\}$ and $V(F')\setminus
V_1=\{v_4\}$ or $V_1=\{v_1,v_3,v_4,v_5,v_6\}$ and
$V(F')\setminus V_1=\{v_2\}$.
Hence, $\mathit{pd}(F')=2$.

\emph{Necessity.} Suppose that $\mathit{pd}(G)=2$.
Clearly, there is at most one cut-vertex since
$D(G)=2$. Otherwise $D(G)\geq3$. We now discuss
it by two cases.

\emph{Case 1.} Suppose that there exists exactly
one cut-vertex. Then the remaining vertices are
adjacent to the cut vertex. Clearly, it follows
that $G\in\mathcal{D}$.

\emph{Case 2.} Suppose that there is not a cut
vertex. Then $\delta(G)\geq 2$. Let $r$ be a
vertex with maximum degree and $N(r)=\{x_1,x_2,\cdots,x_\Delta\}$.

\emph{Subcase 2.1.} $d(r)=n-1$.
Since there is no cut vertex in $G$, the
induced subgraph $G[x_1,\cdots,x_{n-1}]$ is
connected. We claim that $G[x_1,\cdots, x_{n-1}]$
is a path. Otherwise, it is a tree with a vertex
$v$ of degree at least three, or it contains a
cycle. Thus, $G$ contains a minor of $K_{2,3}$ or $K_4$, which is a
contradiction by Theorem \ref{defineouterplanar}.
Clearly, it follows that $G\in \mathcal{D}$.

\emph{Subcase 2.2.} $d(r)=n-2$.
Let $x$ be not adjacent to $r$.
Suppose that $|N(x)|\geq3$. Then, there is a minor of $K_{2,3}$ in $G$. A contradiction. Thus, $|N(x)|=2$.
We now illustrate our claim that $G\cong F^-_{1,5}$
or $F^-_{1,4}$ or $G\cong F'$.
Without loss of generality, let $x_1x,x_2x$
be two edges of $G$. Suppose that $n\geq7$. Since
$D(G)=2$, the vertices $x_3,x_4,x_5$ are adjacent to $x_1$ or
$x_2$. Then there are at least two vertices adjacent
to the same vertex $x_1$ (or $x_2$). Then we obtain
a minor of $K_{2,3}$ in $G$.
Suppose that $n=5$. Since $d(x_3)\geq 2$, we have
exactly one edge $x_2x_3\in E(G)$ or $x_1x_3\in E(G)$.
Otherwise, there is a minor of $K_4$. So
$G\cong F^-_{1,4}$. Suppose that $n=6$. The vertices $x_3$, $x_4$ have
no common neighbor. Otherwise, there
is a minor of $K_{2,3}$ in $G$. Since $D(G)=2$, there exist
edges $x_2x_3$, $x_1x_4$ (or $x_1x_3$, $x_2x_4$).
Hence, $G\cong F^-_{1,5}$ or $G\cong F'$.

\emph{Subcase 2.3.} $d(r)\leq n-3$. Let $y_1$ and
$y_2$ be two vertices which are nonadjacent to $r$.
Assume that $y_1y_2\in E(G)$.
If $|N(r)\cap N(\{y_1,y_2\})|\geq3$, then there
is a minor of $K_{2,3}$ in $G$.
Thus, $|N(r)\cap N(\{y_1,y_2\})|\leq2$.
If $N(r)\cap N(\{y_1,y_2\})=\emptyset$, then
$d(r,y_{1})\geq 3$. A contradiction. If
$|N(r)\cap N(\{y_1,y_2\})|=1$, then when there
is exactly one vertex of $\{y_1,y_2\}$ which is
adjacent to one vertex of $N(r)$, it contradicts
with $D(G)=2$. When both $y_1$ and $y_2$ are
adjacent to one vertex of $N(r)\cap N(\{y_1,y_2\})$,
there exists $x_i$ $(i\neq1)$
which is not adjacent to $x_1$ since $r$ is a vertex with maximum degree. Then there exists
$y_3\in V(G)\setminus (N[r]\cup \{y_1,y_2\})$ such
that $y_3x_i,y_3y_1\in E(G)$. Regard $y_3$ as $y_2$. Then
$|N(r)\cap N(\{y_1,y_2\})|=2$. Thus we only
consider the case that $|N(r)\cap N(\{y_1,y_2\})|=2$.
Without loss of generality, let $x_1y_1,x_2y_2\in E(G)$.

When $|N(r)|=2$, clearly, $G\cong C_5$, contradicting
that $pd(G)=2$. If $|N(r)|\geq 5$, then $x_3,x_4$ and $x_5$
belong to $N(r)$ but not $N(\{y_1,y_2\})$.
So, there are at least two vertices of
$x_3,x_4$ and $x_5$ adjacent to one vertex of $x_1$
and $x_2$, which induces a minor of $K_{2,3}$.
Thus, $|N(r)|=3$ or $4$.
If there exists a vertex $y\notin N(r)$ ($y\neq y_1,y_2$) such that
$x_3y \in E(G)$, then it contains a minor of $K_{2,3}$. If $|N(r)|=3$,
then $N(x_3)\cap N(\{y_1,y_2\})\neq \emptyset$ since
$D(G)=2$. However, there is at most one vertex in $N(r)\cap N(\{y_1,y_2\})$ adjacent
to $x_3$. Otherwise, there is a minor of $K_4$ in $G$.
Suppose that $x_3$ is adjacent
to one vertex of $\{x_1,x_2\}$, say $x_1$, then $D(G)=3$
since $r$ has the maximum degree. A contradiction.
If $|N(r)|=4$, then $x_3$ and $x_4$ are not
adjacent to one common vertex of $x_1$ and $x_2$.
Otherwise, it produces a minor of $K_{2,3}$.
Then let $x_2x_3,x_1x_4\in E(G)$. In the sake of
$D(G)=2$, the pair of vertices $x_3$ and $y_1$
has at least one common neighbor and so does the
pair of $x_4$ and $y_2$. However, it produces a
minor of $K_4$.

Assume that $V(G)\setminus N[r]$ is an independent set.
Clearly, $y_1$ and $y_2$ have at most two neighbors
in $\{x_i| 1\leq i\leq n-3\}$, respectively. Since
$D(G)=2$, $y_1$ and $y_2$ have at least one common
neighbor in $\{x_i| 1\leq i\leq n-3\}$. If they have
at least two common neighbors, then $G$ contains
a minor of $K_{2,3}$, which is a contradiction.
Thus, $y_1$ and $y_2$ have exactly one common
neighbor in $\{x_i| 1\leq i\leq n-3\}$. Without
loss of generality, let $x_1,x_2\in N(y_1)$ and
$x_1\in N(y_2)$. Then $y_2$ has a neighbor $x$
where $x\neq x_1,x_2$, without loss of generality,
let $x=x_3$, then there is an edge $x_1x_3$ or
$x_2x_3$. Otherwise, it is a contradiction to
$D(G)=2$. If $x_2x_3\in E(G)$, then there is a
minor of $K_4$. If $x_1x_3\in E(G)$, then there
is the edge $x_1x_2$ since $D(G)=2$. Since
$d(r)\geq d(x_1)\geq 4$, there exists another
neighbor $x_4$ of vertex $r$, which is not adjacent
to $x_1$. Otherwise, it produces a minor of $K_{2,3}$.
In view of $D(G)=2$, then $x_2x_4$, $x_3x_4\in E(G)$.
There is a minor of $K_4$, a contradiction.
\end{proof}

\subsection{An upper bound and an extremal problem}

We first consider the upper bound of the proper
disconnection number for a graph of order $n$ and chromatic
index $\chi'(G)$.

\begin{thm}\label{pd24}
If $G$ is a nontrivial connected graph, then $pd(G)\leq \chi'(G)-1$.
\end{thm}

\begin{proof}
By Vizing's Theorem, define a proper edge
coloring $c$: $E(G)\rightarrow \chi'(G)$.
Then we redefine an edge-coloring $c'$ of
$G$ as follows: For any edge $e$ of $G$,
if $c(e)=\chi'(G)$, then $c'(e)=1$;
otherwise, $c'(e)=c(e)$. Let $x$, $y$ be
two vertices of $G$. Assume
$N(x)=\{v_1,v_2,\cdots,v_{d(x)}\}$.
Obviously, at most two incident edges
of $x$ are assigned the color $1$. If
there exists at most one incident edge of
$x$ with color $1$, then the edges incident
with $x$ form an $x-y$ proper cut. If there
exist two edges with color $1$, then we may
assume $c'(xv_i)=c'(xv_j)=1\ (1\leq i<j\leq d(x))$.
If $y\in\{v_i,v_j\}$ and let
$t=\{v_i,v_j\}\setminus y$, then
$E_x\cup E_t\setminus\{xt\}$ is an $x-y$
proper cut. Otherwise,
$E_x\cup E_{v_i}\setminus\{xv_i\}$ is an
$x-y$ proper cut. Thus, $c'$ is a proper
disconnection coloring of $G$ and so
$pd(G)\leq \chi'(G)-1$.
\end{proof}

According to Theorem \ref{t2} and Theorem
\ref{pd24}, we get the following result.

\begin{thm}\label{pd25}
 Let $G$ be a nontrivial connected graph
 of order $n$. Then
 $pd(G)\leq min\{ \chi'(G)-1, \left \lceil \frac{n}{2} \right \rceil\}$,
 and the bound is sharp.
\end{thm}

\begin{proof}
By Theorem \ref{pd24}, $pd(G)\leq\chi'(G)-1$.
Since $G$ is a connected subgraph of $K_n$,
$pd(G)\leq \left \lceil \frac{n}{2} \right \rceil$
by Theorem \ref{t2}. For the sharpness, $\left \lceil \frac{n}{2} \right \rceil$
can be reached by complete graphs, and $\chi'(G)-1$
can be reached by even cycles and paths with at least 3 vertices.
\end{proof}

Now we investigate the following extremal problem:
For given $k, n$ of positive integers with
$1\leq k\leq \lceil\frac{n}{2}\rceil$, what
is the minimum possible size of a connected
graph $G$ of order $n$ such that the proper
disconnection number of $G$ is $k$.

\begin{lem}\label{match}
Let $G$ be a connected graph of order $n$.
Let $M$ be a set of matching edges.
Then $pd(G)\leq  \max\{pd(G_i)|1\leq i\leq t\}+1$,
where $G_i$ is a connected
component  and $t$ is the number of components of $G-M$.
\end{lem}
\begin{proof}
Let $\{G_1,G_2, \ldots, G_t\}$ be $t$
components of $G-M$ and let
$\ell=\max\{\textnormal{pd}(G_i)| 1\leq i\leq t\}$.
Let $c_i$ be a pd-coloring of $G_i$.
Let $F_{G_i}(u,v)$ be $u, v$ proper cut in $G_i$.
We define an edge coloring $c$:
$E(G)\rightarrow[\ell+1]$ of $G$ by
$c(e)=c_i(e)$ if $e\in E(G_i)$ and
$c(e)=\ell+1$ if $e\in M$.
Let $x,y$ be two vertices of $G$.
If $x,y\in G_i$, then $F_{G_i}(x,y)\cup M$ is
an $x,y$ proper cut in $G$.
If $x\in G_i$, $y\in G_j$ where $i\neq j$, then
$M$ is an $x,y$ proper cut in $G$.
Hence, $pd(G)\leq \max\{pd(G_i)|1\leq i\leq t\}+1$.
\end{proof}

\begin{thm}
For integers $k$ and $n$ with
$1\leq k\leq \lceil\frac{n}{2}\rceil$,
the minimum size of a connected graph $G$ of
order $n$ with $\textnormal{pd}(G)=k$ is
\begin{equation*}
|E(G)|_{min}=
\left\{
  \begin{array}{ll}
    n-1, & \hbox{ if  $k=1$,}\\
    n+2k-4, & \hbox{ if $k\geq 2$. }\\

    \end{array}
\right.
\end{equation*}
\end{thm}

\begin{proof}
For $k=1$, $|E(G)|\geq n-1$ since $G$ is
connected and $pd(T)=1$ by Theorem \ref{pd23}.
Thus, $|E(G)|_{min}=n-1$.

For $k\geq 2$,
we first show that if the size of a connected
graph $G$ of order $n$ is at most $n+2k-5$,
then $pd(G)\leq k-1$.
We proceed by induction on $k$. We have seen
that the result is true for $k=2$ by Proposition
\ref{tree}.
Suppose that $G$ is a graph with $e(G)\leq n+2k-5$.
If $G$ is a graph with at most one
block which is a cycle and other blocks are trivial,
the result is true for $G$ by Lemma \ref{cycle} and Lemma \ref{5}.
Otherwise,
we claim that there exist two matching edges,
denoted $e_1$, $e_2$, of $G$ such that
$G\setminus\{e_1,e_2\}$ is connected graph of order $n$.
Now, we distinguish two cases as follows:

($i$) $G$ has exactly one block with at least
$4$ vertices which is not a cycle, and other blocks
are trivial;

($ii$) $G$ has at least two blocks with at
least $3$ vertices respectively.

For ($i$), let $B_k$ be a block with at least
$4$ vertices which is not a cycle. Then $B_k$ contains
two vertices $x$ and $y$ such that they are connected
by at least three internally-disjoint $(x,y)$-paths.
We can respectively pick one edge from
two paths of $(x,y)$-paths as matching edges;
For ($ii$), we can respectively pick one edge from
two blocks $B_i$ and $B_j$ as matching edges. The
edges from $(i)$ and $(ii)$ can insure that
$H=G\setminus\{e_1,e_2\}$ is also connected of order $n$.
Since $e(H)\leq n+2(k-1)-5$, we have pd$(H)\leq k-2$
by induction hypothesis.
Then $pd(G)=pd(H+\{e_1,e_2\})\leq pd(H)+1\leq k-1$ by
Lemma \ref{match}.
Hence, we obtain that if $\textnormal{pd}(G)=k$,
then $|E(G)|\geq n+2k-4$.

Next we show that for each pair integers $k$ and $n$ with
$2\leq k\leq \lceil\frac{n}{2}\rceil$ there is a connected graph $G$ of
order $n$ and size $n+2k-4$ such that $pd(G)=k$.
Let $H=K_{2,2k-3}$ with bipartition $A=\{a_1,a_2\}$ and
$B=\{b_1,b_2,\ldots, b_{2k-3}\}$. Let $G$ be the graph
of order $n$ and size $n+2k-4$ obtained from $H$ by
adding an edge $a_1a_2$ and adding $n-2k+1$ pendent
edges to a vertex of $H$.
We obtain that $pd(G)=\textnormal{pd}(H+a_1a_2)= k$
by Lemma \ref{5}.
\end{proof}

\end{document}